\documentclass[reqno]{amsart}

\usepackage{mathrsfs, amsthm}
\usepackage{amssymb}
\usepackage{bbm,enumerate}
\usepackage{fullpage}
\usepackage{fontenc}

\numberwithin{equation}{section}

\newtheorem{thm}{Theorem}[section]
\newtheorem{lemma}[thm]{Lemma}

\newcommand\A{\mathcal A}
\newcommand\B{\mathcal B}
\newcommand\CC{\mathcal C}
\newcommand\I{\mathcal I}
\newcommand\J{\mathcal J}
\newcommand\C{\mathbb C}
\newcommand\Hom{\operatorname{Hom}}
\theoremstyle{definition}
\newtheorem{example}[thm]{Example}

\begin{document}

\title{Woronowicz's Tannaka-Krein duality and free orthogonal quantum groups}

\author{Sara Malacarne}

\email{saramal@math.uio.no}

\address{Department of Mathematics, University of Oslo, P.O. Box 1053
Blindern, NO-0316 Oslo, Norway}

\thanks{Supported by the European Research Council under the European Union's Seventh
Framework Programme (FP/2007-2013)/ ERC Grant Agreement no. 307663 (PI: S. Neshveyev)}

\date{February 15, 2016}

\begin{abstract}
Given a finite dimensional Hilbert space $H$ and a collection of operators between its tensor powers satisfying certain properties, we give a category-free proof 
of the existence of a compact quantum group $G$ with a fundamental representation $U$ on $H$ such that the intertwiners between the tensor powers of~$U$ 
coincide with the given collection of operators. We then explain how the general version of Woronowicz's Tannaka-Krein duality can be deduced from this.
\end{abstract}

\maketitle

\section*{Introduction}

The aim of the paper is to give a category-free proof of Woronowicz's Tannaka-Krein duality Theorem \cite{WTK}. We consider a finite dimensional Hilbert space $H$ and a collection 
of operators between its tensor powers satisfying certain properties. Categorically speaking, we deal with a C$^*$-tensor category  with conjugates that is a subcategory of the category of finite 
dimensional Hilbert spaces, $\text{Hilb}_{f}$, and assume that such category is generated by one self-conjugate Hilbert space. We prove the existence of a compact quantum group $G$, such that its 
representation category  $\operatorname{Rep}G$ is our given category. The proof consists of an explicit 
reconstruction of the Hopf $*$-algebra $\mathbb{C}[G]$, sometimes denoted by $\operatorname{Pol}G$, generated by the coefficients of all finite dimensional representations of $G$. The 
relations defining such Hopf $*$-algebra are directly obtained  through morphisms in the category, or equivalently, through the collection of operators between tensor powers of $H$. The version of  
Woronowicz's Tannaka-Krein Theorem that we prove is essentially formulated in the paper by T. Banica and R. Speicher \cite{BS}, where the duality is used for the construction of new examples of free
quantum groups, via 
subcategories of the so-called categories of noncrossing partitions. Even though the proof here presented is, in many respects, similar to the proofs of Woronowicz's Tannaka-Krein duality appearing in 
\cite{WTK} and \cite{NT}, we wish to point out that this version is more algebraic, mostly category-free and the key part is based on simple duality statements for finite dimensional 
vector spaces. Related to this reconstruction process,  it is also important to mention P. Schauenburg's paper \cite{PS}, in which  a proof of Tannaka-Krein duality is given in a more general 
setting:  monoidal categories that are not semisimple are considered and, correspondingly, arbitrary Hopf algebras are recovered. 

The paper is structured in the following way:
in Section 1 we define a bialgebra, which we will later prove to be the Hopf $*$-algebra $\mathbb{C}[G]$. The relations 
defining this bialgebra are obtained from the collection of operators between tensor powers of $H$, denoted by $\mathcal{C}$.
In Section 2 we prove that the bialgebra defined in Section 1 can be equipped with a Hopf $*$-algebra structure. For this we first consider a smaller collection of operators, $\mathcal{C}_{F}$,
and show that they define the free orthogonal quantum group $O_F^+$, \cite{VDW}. We want to stress that nothing,
apart from the fact that $\mathbb{C}[O_{F}^{+}]$ is a well-defined Hopf $*$-algebra, is used.
In Section 3  we prove the equivalence between the C$^*$-tensor category generated by one selfdual Hilbert space $H$  and $\operatorname{Rep}G$. In Section 4 we show how the particular case analysed 
in Section 3 can be extended to the general case of a not necessarily finitely generated C$^*$-tensor category.

\bigskip

 \textbf{\textit{Acknowledgments.}} I would like to thank my supervisor Sergey Neshveyev for his help and precious advice throughout this work.
 I am grateful to Teodor Banica for useful suggestions. Thanks also to Marco Matassa for fruitful discussions.

\section{Singly generated categories of Hilbert spaces}

Our goal is to prove the following version of Woronowicz's Tannaka-Krein duality.

\begin{thm}\label{tmain}
Let $H$ be a finite dimensional Hilbert space. Suppose we are given a collection $\CC$ of spaces $\CC(k,l) $ of operators $H^{\otimes k}\to H^{\otimes l}$ for all $k,l\ge0$ satisfying the following 
properties:
\begin{itemize}
\item[{\rm (1)}] if $T,S\in\CC$, then $T\otimes S\in\CC$;
\item[{\rm (2)}] if $T,S\in\CC$ are composable, then $TS\in\CC$;
\item[{\rm (3)}] $T\in\CC$ implies $T^*\in\CC$;
\item[{\rm (4)}] $\CC(k,k)$ contains the identity operator for all $k\ge0$;
\item[{\rm (5)}] $\CC(0,2)$ contains an operator $R$ such that $(R^*\otimes \iota)(\iota\otimes R)=\pm\iota$ on $H$.
\end{itemize}

Then there exists a unique up to isomomorphism compact quantum group $G$ with a self-conjugate fundamental representation $U$ on $H$ such that $\Hom_G(H^{\otimes k},H^{\otimes l})=\CC(k,l)$ for all 
$k,l\ge0$.
\end{thm}

In the last section we will discuss how the general form of Woronowicz's Tannaka-Krein duality can be easily deduced from this.

\smallskip

Denote by $\mathcal{A}$ the tensor algebra of the space of linear functionals on $B\left(H\right)$,
 i.e., $$\mathcal{A}:=T\left(B\left(H\right)^{*}\right)=\bigoplus\limits_{k=0}^{\infty} B\left(H^{\otimes k}\right)^{*}.$$  Let $U \in B\left(H\right) \otimes B(H)^*\subset B\left(H\right) \otimes 
 \mathcal{A}$ be the ``fundamental matrix'' of $\mathcal{A}$, so $U$ is characterized  by the property
\begin{equation} \label{efund}
(\iota\otimes T)(U)=T\ \ \forall T\in B(H).
\end{equation}
In other words,
$$U:= \sum\limits_{i,j} e_{ij} \otimes u_{ij},$$
where the $e_{ij}$'s are matrix units in $B\left(H\right)$ and $\{u_{ij}\}_{i,j}$ is the dual basis of $B\left(H\right)^{*}$ such that $u_{ij}\left(e_{kl}\right)=\delta_{ik}\delta_{jl}$. The tensor 
algebra $\mathcal{A}$ is a bialgebra  with comultiplication $\Delta$ defined by duality from the multiplication on $B(H)$, so that
$$\Delta\left(u_{ij}\right)=\sum\limits_{k} u_{ik}\otimes u_{kj},$$
or equivalently, using the leg-numbering notation, $(\iota \otimes \Delta)(U)=U_{12}U_{13}$.

Next, denote by $\mathcal{A}_{n}\subset\A$ the subspace given by
$$\mathcal{A}_{n}:=\bigoplus\limits_{k=0}^{n} B\left(H^{\otimes k}\right)^{*} = \left(\bigoplus\limits_{k=0}^{n} B\left(H^{\otimes k}\right)\right)^{*},$$ and denote by $\mathcal{B}_{n}$ the commutant
$$\mathcal{B}_{n} := \left( \bigoplus\limits_{k,l=0}^{n} \CC(k,l)
\right)^{\prime}\subseteq \bigoplus_{k=0}^{n} B(H^{\otimes k})\subset B\left(\bigoplus_{k=0}^{n} H^{\otimes k}\right).$$

Finally, let $$\mathcal{I}_{n} := \left\lbrace a \in \mathcal{A}_{n} \; : \; \left.a\right|_{\mathcal{B}_{n}}=0\right\rbrace,$$ and denote by $\mathcal{I}$ the union
$\mathcal{I}:= \bigcup_{n=0}^{\infty} \mathcal{I}_{n}$. Note that $\mathcal{I}_{n+1} \cap \mathcal{A}_{n}=\mathcal{I}_{n}$, so $\mathcal{I}$ is a subspace of $\A$.

\begin{lemma} \label{lem1}
$\mathcal{I}$ is a bi-ideal in the bialgebra $\mathcal{A}$.
\end{lemma}

\begin{proof}
We will first prove that $\mathcal{I}$ is an ideal. 
Assume $a\in \mathcal{I}_{n}$ and $b\in B(H^{\otimes m})^{*}$;$\;$ we have to check that $a\otimes b$ vanishes on $\left(\oplus_{k,l=m}^{n+m}\CC(k,l)\right)^{\prime}$. Since
$\left(\oplus_{k,l=m}^{n+m}\CC(k,l)\right)^{\prime}\subseteq \left( \oplus_{k,l=0}^{n}\CC(k,l) \right)^{\prime}\otimes \left( \oplus_{k,l=0}^{m}\CC(k,l) \right)^\prime$, the statement simply follows 
from 
the assumption that $a$ vanishes on $\left( \oplus_{k=0}^{n}\CC(k,l) \right)^{\prime}$. 

To prove that $\mathcal{I}$ is a coideal we have to show that $\Delta(\mathcal{I}) \subseteq 
\mathcal{I}\otimes \mathcal{A}+\mathcal{A}\otimes \mathcal{I}$.  For this purpose we use an equivalent definition of $\mathcal{I}$, that is, we consider the space spanned by  the slices
\begin{equation}\label{eslice}
(\omega \otimes \iota) \left( (T\otimes 1)U^{\otimes k}-U^{\otimes l}(T \otimes 1)\right)
\end{equation}
 for all $\omega \in B(H^{\otimes k}, H^{\otimes l})^{*}$, $T \in \CC(k,l)$ and $k,l\ge0$. This space indeed coincides with $\I$, since using \eqref{efund}
 we see that an operator $S\in\oplus_{k=0}^{n} B(H^{\otimes k})$ vanishes on the elements \eqref{eslice} for all $k,l\le n$ if and only it lies in $\B_n$.
We choose an orthonormal basis of $H^{\otimes k}$, $\{\xi_{j}\}_{j}$, and of $H^{\otimes l}$, $\{\eta_{i}\}_{i}$, and assume $\omega$ is of the form
$\omega_{ij}=\left\langle\; \cdot\; \xi_{j},\eta_{i}\right\rangle$. We set $V:=U^{\otimes k}$ and $W:=U^{\otimes l}$. Then, using the leg-numbering notation,
$$\Delta(\omega_{ij}\otimes \iota)\left((T\otimes 1)V-W(T\otimes 1)\right)$$
is equal to $$(\omega_{ij} \otimes \iota\otimes \iota)\left((T\otimes 1 \otimes 1)V_{12}V_{13}-W_{12}W_{13}(T\otimes 1 \otimes 1)\right).$$
The expression in the parentheses can be written as $$\left((T\otimes 1 \otimes 1)V_{12}- W_{12}(T \otimes 1 \otimes 1)\right)V_{13}+W_{12}
\left((T\otimes 1 \otimes 1)V_{13}-W_{13}(T\otimes 1 \otimes 1)\right).$$
Now, if we just consider the first part of the sum
$$(\omega_{ij}\otimes \iota\otimes \iota)\left(\left((T\otimes 1 \otimes 1)V_{12}- W_{12}(T \otimes 1 \otimes 1)\right)V_{13}\right),$$
it can be expressed as $$\sum_{k} (\omega_{ik}\otimes \iota\otimes \iota) \left((T\otimes 1\otimes 1)V_{12}-W_{12}(T\otimes 1\otimes 1) \right)\;(\omega_{kj}\otimes \iota\otimes \iota)(V_{13}),$$
which belongs to $\mathcal{I}\otimes \mathcal{A}$. Similarly the other part lies in $\mathcal{A}\otimes \mathcal{I}$, as wanted.

\end{proof}

By the previous Lemma, $\mathcal{A}/\mathcal{I}$ is a bialgebra. What we wish to prove is that $\mathcal{A}/\mathcal{I}\cong \mathbb{C}[G]$,
for a compact quantum group $G$, and to do so we need for $\mathcal{A}/\mathcal{I}$ to be a Hopf $*$-algebra and for $U$ to be unitary (see Theorem 1.6.7 of \cite{NT}),
and this is not obvious written in this manner. In fact it is not even clear whether $\mathcal{A}/\mathcal{I}$
has a $*$-structure. We shall proceed with an intermediate step. The idea is the following: we will introduce another bi-ideal in $\mathcal{A}$, $\mathcal{I}_{F}$,
and show that $\mathcal{A}/\mathcal{I}_{F} \cong \mathbb{C}[O_{F}^{+}]$, where $O_{F}^{+}$ is the free orthogonal quantum group. Thus, $\mathcal{A}/\mathcal{I}_{F}$ will automatically inherit a Hopf 
$*$-algebra structure. Finally we will show that $\mathcal{I}/\mathcal{I}_{F}$ is a Hopf $*$-ideal in $\mathcal{A}/\mathcal{I}_{F}$ and then conclude that there exists a compact quantum group $G$ 
such that $$\mathcal{A}/\mathcal{I} \cong \mathcal{A}/\mathcal{I}_{F} \Big/ \mathcal{I}/\mathcal{I}_{F} \cong \mathbb{C}[G],$$ again by Theorem 1.6.7 of \cite{NT}.

\bigskip

\section{Representation category of a free orthogonal quantum group}

Following the strategy described above, we now consider the case when $\CC$ is the smallest collection of spaces as in Theorem~\ref{tmain} containing a fixed operator $R\colon\C\to H^{\otimes 2}$ such 
that $(R^*\otimes\iota)(\iota\otimes R)=\pm\iota$. It is known, and not difficult to see, that if we fix an orthonormal basis $e_1,\dots,e_n$ in $H$, then $R$ has the form $(\iota\otimes F)r$, where 
$r\colon\C\to H\otimes H$ is given by $r(1)=\sum_i e_i\otimes e_i$ and $F\in GL_{n}(\mathbb{C})$ is such that $F\bar{F}=\pm1$, where $\bar F$ is the matrix obtained from $F$ by taking the complex 
conjugate of every entry. We will use the subindex $F$ for the constructions of the previous section related to this smallest collection, so we write $\CC_F$, $\B_{F,n}$, $\I_{F,n}$, etc.

Consider the universal unital algebra $\C[O_{F}^{+}]$ generated by entries of a matrix $U=(u_{ij})_{i,j}$ satisfying the relations
$$
UF^{t}U^{t}(F^{-1})^{t}=1,\ \ F^{t}U^{t}(F^{-1})^{t}U=1.
$$
for an invertible $n$ by $n$ matrix $F$. It is again known and easy to see that this is a Hopf $*$-algebra with comultiplication $\Delta(u_{ij})=\sum_{k} u_{ik}\otimes u_{kj}$ and involution given by
$U^*=F^{t}U^{t}(F^{-1})^{t}$. The compact quantum group $O_{F}^{+}$ thus defined is known in literature as the free orthogonal quantum group \cite{VDW}, but we do not need to know any properties of 
this quantum group apart from the fact that it is well-defined. The following Lemma is a simple consequence of our definitions.

\begin{lemma} \label{lorth}
We have $\mathcal{A}/\mathcal{I}_{F} \cong \mathbb{C}[O^{+}_{F}]$.
\end{lemma}

\begin{proof}
By definition, the  bialgebra $\mathbb{C}[O^{+}_{F}]$  can be written as  $\mathcal{A}/\mathcal{L}$ where $\mathcal{L}$ is the ideal generated by the elements
 $$\left(F^{t}-(UF^{t}U^{t})\right)_{ij},\ \ \left((F^{-1})^{t}-U^{t}(F^{-1})^{t}U\right)_{ij}\ \ \ \forall i,j    $$
  with $U\in B(H)\otimes \mathcal{A}$ being the fundamental matrix of $\mathcal{A}$.
In order to prove the Lemma we need to show that $\mathcal{I}_{F}=\mathcal{L}$. To show that $\mathcal{L}\subseteq \mathcal{I}_{F}$, consider the linear functionals
$\omega_{1,ij}:=\left\langle\;\cdot\;1,e_{i}\otimes e_{j}\right\rangle\in B(\mathbb{C},H^{\otimes 2})^{*}$ and $\omega_{ij,1}:=\left\langle\;\cdot\;e_{i}\otimes e_{j},1\right\rangle \in
B(H^{\otimes 2},\mathbb{C})^{*}$. Then
$$\left(F^{t}-(UF^{t}U^{t})\right)_{ij}=(\omega_{1,ij}\otimes \iota)\left((R\otimes \iota)-U^{\otimes 2}(R\otimes \iota)    \right)$$
and
$$ \left((F^{-1})^{t}-U^{t}(F^{-1})^{t}U \right)_{ij}= \pm(\omega_{ij,1}\otimes \iota)\left((R^{*}\otimes \iota)-  (R^{*}\otimes \iota)U^{\otimes 2}  \right),  $$
where we recall that $R=(\iota \otimes F)r$ and in the second equality we use that $F^{*}=\pm(F^{-1})^{t}$. Hence $\mathcal{L}\subseteq \mathcal{I}_{F}$.

Conversely, let us show that $\mathcal{I}_{F}\subseteq \mathcal{L}$. As follows from the above identities, $R$ and $R^{*}$ are morphisms in the category $\operatorname{Rep}O_{F}^{+}$. It follows that 
any operator in $\CC_F$ is a morphism in $\operatorname{Rep}O_{F}^{+}$. But this implies $\mathcal{I}_{F}\subseteq \mathcal{L}$ because any
relation defined by elements of $\mathcal{C}_{F}$  has to be satisfied  in $\mathbb{C}[O_{F}^{+}]$.

\end{proof}

Therefore in order to prove Theorem~\ref{tmain} for $\CC=\CC_F$ it remains to establish the following Lemma. We remark that the Lemma itself is not needed for the general case, but its proof will be 
reused.

\begin{lemma} \label{lorth2}
We have $\Hom_{O^+_F}(H^{\otimes k},H^{\otimes l})=\CC_F(k,l)$ for all $k,l\ge0$.
\end{lemma}

\begin{proof}
We have to show that for every $n$
 $$\bigoplus\limits_{k,l=0}^{n}\CC_F(k,l)=\text{End}_{O_{F}^{+}}\left( \bigoplus\limits_{k=0}^{n}H^{\otimes k}\right).$$
As both sides are (finite dimensional) von Neumann algebras, the equality is equivalent to the equality of their commutants:
$$\B_{F,n}=\text{End}_{O_{F}^{+}}
\left( \bigoplus\limits_{k=0}^{n}H^{\otimes k}\right)^{\prime}.$$

Recall now that for
any finite dimensional representation $V\in B(H_{V})\otimes \mathbb{C}[G]$ of a compact quantum group~$G$ we have a representation $\pi_V$ of the algebra $\C[G]^*$ on 
$H_V$ defined by $\pi_V(\phi)=(\iota\otimes\phi)(V)$, and then $\pi_{V}(\mathbb{C}[G]^{*})=\text{End}_{G}(H_{V})^{\prime}$. Therefore we have to show that
$\pi_{\bigoplus_{k=0}^{n} U^{\otimes k}}\left(\mathbb{C}[O_{F}^{+}]^{*}\right)=\mathcal{B}_{F,n}$.
But this immediately follows from the previous Lemma, as $\C[O^+_F]=\mathcal{A}/\mathcal{I}_{F}\supset\mathcal{A}_{n}/\mathcal{I}_{F,n}=\mathcal{B}^{*}_{F,n}$.

\end{proof}

\bigskip

\section{Proof of the Theorem}

We now turn to a general $\CC$ as in Theorem~\ref{tmain}. Let $R\in\CC(0,2)$ be an operator such that $(R^*\otimes\iota)(\iota\otimes R)=\pm\iota$. As in the previous section, we fix an orthonormal 
basis $e_1,\dots,e_n$ in $H$ and write $R$ as $(\iota\otimes F)r$. Denote by $\J$ the bi-ideal $\I/\I_F$ in $\C[O_F^+]=\A/\I_F$. Note that $\J$ can still be described as the space spanned by the 
elements
$ (\omega\otimes \iota)\left((T\otimes 1)U^{\otimes k}-U^{\otimes l}(T \otimes 1)\right)$, for $T\in \CC(k,l)$ and $\omega\in B(H^{\otimes k},H^{\otimes l})^*$, where we slightly abuse the notation 
and denote by the same symbol $U$ the fundamental matrix of $\A$ and its image in $\C[O_F^+]$.

\begin{lemma}
 $\mathcal{J}$ is a Hopf $*$-ideal in $\mathbb{C}[O_{F}^{+}]$.
\end{lemma}

 \begin{proof}
We denote by $S$ the antipode of $\mathbb{C}[O_{F}^{+}]$. Since $\J$ is a bi-ideal, we only need to check that $\J$ is closed under taking the adjoints and is invariant under $S$. Let $a^{*}=(\omega
\otimes \iota)\left((T\otimes 1)U^{\otimes k}-U^{\otimes l}(T \otimes 1)\right)^{*}$.

We have to show that it lies in $\mathcal{J}$ for any $k,l\ge0$. We first note that $\CC(1,1)$ is closed 
under the operation $\vee$ defined by
$$
(\iota\otimes T)R=(T^\vee\otimes\iota)R,\ \ \text{since}\ \ T^\vee=\pm(\iota\otimes R^*)(\iota\otimes T\otimes\iota)(R\otimes\iota).
$$
As $(\iota\otimes T)r=(T^t\otimes\iota)r$, we have $T^\vee=(F^{-1}TF)^t$, and the inverse operation, still preserving $\CC(1,1)$, is $T\mapsto FT^tF^{-1}$. 
We recall from the previous section that we also have $U^*=F^{t}U^{t}(F^{-1})^{t}$. Analogous formulas hold for $T\in\CC(k,l)$. In fact, if we denote by $F_{k}=F^{\otimes k}$ and $U_{k}=U^{\otimes k}$,
 
then $T^{\vee}=(F_{l}^{-1}TF_{k})^t\in \CC(l,k)$ and  $U_{k}^{*}=F_{k}^{t}U_{k}^{t}(F_{k}^{-1})^{t}$. Therefore, choosing an orthonormal basis of $H^{\otimes k}$, $\{\xi_{j}\}_{j}$, and of
$H^{\otimes l}$, $\{\eta_{i}\}_{i}$, and assuming $\omega$ is of the form $\omega_{ij}=\left\langle\; \cdot\; \xi_{j},\eta_{i}\right\rangle$, we have
\begin{equation*}
\begin{split}
a^{*}=&(\omega_{ij}\otimes \iota)\left((T\otimes 1)U_{k}-U_{l}(T\otimes 1)\right)^{*}=\\
&(\omega_{ji}\otimes \iota)\left(U_{k}^{*}(T^{*}\otimes 1)-(T^{*}\otimes 1)U_{l}^{*}\right)=\\
&(\omega_{ji}\otimes \iota) \left(\left(((T^{*})^{t}\otimes 1)(U_{k}^{*})^{t}-(U_{l}^{*})^t((T^{*})^{t}\otimes 1)\right)^{t}\right)=\\
&(\omega_{ij}\otimes \iota)\left(((T^{*})^{t}\otimes 1)(U_{k}^{*})^{t}-(U_{l}^{*})^{t}((T^{*})^{t} \otimes 1)\right)=\\
&(\omega_{ij}\otimes \iota)\left(((T^{*})^{t}\otimes 1)(F_{k}^{-1}U_{k}F_{k})-(F_{l}^{-1}U_{l}F_{l})((T^{*})^{t} \otimes 1)\right)=\\
&(\omega_{ij}\otimes \iota)\left((F_{l}^{-1} \otimes 1)\left((F_{l}(T^{*})^{t}F_{k}^{-1}\otimes 1)U_{k}-U_{l}(F_{l}(T^{*})^{t}F_{k}^{-1} \otimes 1)\right)(F_{k} \otimes 1)\right)=\\
&\sum\limits_{m,n} (F_{l}^{-1})_{im}(F_{k})_{nj}\;(\omega_{mn}\otimes\iota) \left( (\tilde{T}\otimes\iota)U_{k}-U_{l}(\tilde{T}\otimes\iota)\right),
\end{split}
\end{equation*}
where $\tilde{T}=F_{l}(T^{*})^{t}F_{k}^{-1} \in \CC(k,l)$, since $T\mapsto F_{l}T^tF_{k}^{-1}$ is a map from $\CC(l,k)$ to  $\CC(k,l)$, being the inverse operation of $\vee$. Hence, $a^{*}\in \J$.

The invariance of $\J$  under the antipode immediately follows from its invariance under involution. If $a=(\omega_{ij}\otimes \iota )\left((T\otimes 1)U_{k}-U_{l}(T \otimes 1)\right)$ then

\begin{equation*}
\begin{split}
S(a)=&(\omega_{ij}\otimes \iota)\left((T\otimes 1)(\iota\otimes S)(U_{k})-(\iota\otimes S)(U_{l})(T\otimes 1)\right)\\
=&(\omega_{ij}\otimes \iota)\left((T\otimes 1)U_{k}^{*}- U_{l}^{*}(T\otimes 1)\right)\\
=&(\omega_{ji}\otimes \iota)\left(U_{k}(T^{*}\otimes 1)- (T^{*}\otimes 1)U_{l}\right)^{*} \in \mathcal{J}.
\end{split}
\end{equation*}

\end{proof}

Given the above Lemma, we conclude that there exists a compact quantum group $G$ such that $\mathcal{A}/\mathcal{I}\cong \mathbb{C}[G]$. It remains to show 
that $\Hom_G(H^{\otimes k},H^{\otimes l})=\CC(k,l)$. But this is done in exactly the same way as in Lemma~\ref{lorth2}. 

To finish the proof of Theorem~\ref{tmain} we have to show that  the compact quantum group $G$ is unique up to isomorphism. Let $G^\prime$ be another compact quantum group satisfying  
the assumptions of Theorem~\ref{tmain}, that is, having a  fundamental representation $V=(v_{ij})_{ij}$ on $H$ such that $\Hom_{G^{\prime}}(H^{\otimes k},H^{\otimes l})=\CC(k,l)$.
We can identify $\C[G^\prime]$ with $\A/\I^\prime$, for a bi-ideal 
$\I^\prime\subset \A$. The only thing to check is that the bi-ideal $\mathcal{I}^\prime$ is completely determined by the operator spaces $\CC(k,l)$. Since $\CC(k,l)$ and $\Hom_{G^\prime}(H^{\otimes k},
H^{\otimes l})$ coincide, from the proof of Lemma~\ref{lorth2} we see that this implies that $\mathcal{A}_{n}/\mathcal{I}^\prime_{n}=\B_{n}^*$, where $\mathcal{I}^\prime_n=\mathcal{I}^\prime \cap 
\mathcal{A}_n$ 
and 
$\mathcal{I}^\prime=\bigcup_{n\geq 0}\; \mathcal{I}^\prime_n$. Hence, the spaces $\mathcal{I}^\prime_n$ are completely determined by the spaces $\CC(k,l)$. Thus Theorem~\ref{tmain} is proved.

\bigskip

\section{General version of the Tannaka-Krein duality}

In this section we want to explain, without too many details, how using Theorem~\ref{tmain} one can recover the following result.

\begin{thm}[Woronowicz's Tannaka-Krein duality]
 Let $\mathcal{C}$ be an essentially small C$^*$-tensor category with conjugates, $\tau:\mathcal{C}\rightarrow \text{Hilb}_{f}$ be a unitary fiber functor. Then
 there exists a compact quantum group $G$ and a unitary monoidal equivalence $\theta: \mathcal{C}\rightarrow \operatorname{Rep}G$ such that
 $\tau$ is naturally unitarily monoidally isomorphic to the composition of the canonical fiber functor $\pi:\operatorname{Rep}G \rightarrow \text{Hilb}_{f}$
 with $\theta$. Furthermore, the Hopf $*$-algebra $(\mathbb{C}[G],\Delta)$ for such a $G$ is uniquely determined up to isomorphism.
\end{thm}

We remark that for C$^*$-tensor categories we follow the conventions of~\cite{NT}, in particular, we assume that they are closed under finite direct sums and subobjects.

\smallskip

We concentrate only on the existence of $G$. We may assume that $\CC$ is a subcategory of $\text{Hilb}_{f}$ and $\tau$ is the embedding functor. If there exist an object $H$ in $\CC$ such that any 
other object is isomorphic to a subobject of $H^{\otimes n}$ for some $n\ge0$, and a morphism $R\colon\C\to H\otimes H$ such that $(R^*\otimes\iota)(\iota\otimes R)=\pm\iota$ , then the result follows 
from Theorem~\ref{tmain}. For general $\CC$ let us distinguish between three cases:
\begin{enumerate}[$(i)$]
\item $\mathcal{C}$ is generated, as a C$^*$-tensor category with conjugates, by one object;
\item $\mathcal{C}$ is generated by a finite number of objects;
\item $\mathcal{C}$ is infinitely generated.
\end{enumerate}

$(i)$ Assume $\mathcal{C}$ is generated by one object $K$, so every object of $\CC$ is isomorphic to a subobject of a tensor product of copies of $K$ and an object $\bar K$ conjugate to $K$. Let $(R',
\bar R')$ be a solution
 of the conjugate equations for $K$ and $\bar{K}$. Then letting $H=K\oplus \bar K$ and $R=R'\oplus\bar R'$, considered as a morphism $\C\to H\otimes H$, we have $(R^*\otimes\iota)(\iota\otimes R)=
 \iota$,
 so we are back to the case covered by Theorem~\ref{tmain}.

$(ii)$ The case when $\CC$ is generated by a finite number of objects $H_1,\dots,H_n$ is not much different from $(i)$, as then $\CC$ is generated by $H_1\oplus\dots\oplus H_n$.

$(iii)$ For general $\CC$, choose a generating set $\mathcal F$ in $\CC$ and let $\mathcal{E}$ be the family of finite subsets of $\mathcal F$ ordered by inclusion. For each $E\in \mathcal{E}$ let 
$\mathcal{C}_{E}$ be the full rigid
C$^*$-tensor subcategory of $\mathcal{C}$ generated by the finite set of objects in $E$. By
the previous case, for each subcategory $\mathcal{C}_{E}$ we get a compact
quantum group $G_{E}$ with representation category $\mathcal{C}_{E}$. Moreover, if $E\subset E'$, then, since $\CC_E\subset\CC_{E'}$, by the uniqueness part of Theorem~\ref{tmain}, the quantum group 
$G_E$ is a quotient of $G_{E'}$, that is, we have an embedding $\C[G_E]\hookrightarrow\C[G_{E'}]$ of Hopf $*$-algebras. Then $\C[G]$ is defined as the inductive limit of the Hopf $*$-algebras $\C[G_E]
$.
\medskip

In the following example we can see how to recover the free unitary quantum group following the procedure explained in point $(i)$ of the above.
\begin{example}[Free Unitary quantum group]
We denote by $\C[U_{Q}^+]$ the universal unital $*$-algebra  generated by the entries of matrices $V=(v_{ij})_{i,j} $ and  $\bar V=(\bar v_{ij})_{i,j}$ such that $V$  and $\bar V$ are unitary with 
involution defined by  $V^*=Q^t \bar V^t (Q^{-1})^t$ and $\bar V^*= (Q^{-1})^* V^t Q^*$, for an invertible $n$ by $n$ matrix $Q$. 
The algebra $\C[U_Q^+]$ is a Hopf $*$-algebra with comultiplication $\Delta(v_{ij})=\sum_k v_{ik}\otimes v_{kj}$ and  $U_Q^+$ is known in literature as the free unitary quantum group.

We wish to prove the equivalence between the representation category of the free unitary quantum group and a concrete C$^*$-tensor category having certain properties. More specifically, 
 consider the  Hilbert space $K=\C^n$ and its complex conjugate $\bar{K}$. Let  $\CC_Q$ be the smallest collection of operators between tensor powers 
of $H:=K\oplus \bar K$, as in Theorem~\ref{tmain}, containing the operator $R\colon\C\to H^{\otimes 2}$ such that $(R^*\otimes \iota)(\iota\otimes R)=\iota$,  and the projection
$p\colon K\oplus \bar K\to K$. The operator $R$ is equal to $(\iota \otimes F)r$ for $F\in GL_{2}(M_n(\C))$ with entries $F_{11}=F_{22}=0$, $F_{12}=\bar Q^{-1}$ and $F_{21}=Q$, where 
$\bar{Q}$  is the matrix whose coefficients are the complex conjugates of the entries of $Q$.

We claim that $\Hom_{U_Q^+}(H^{\otimes k},H^{\otimes l})=\CC_Q(k,l)$ for all $k,l\ge0$. We  show that  $\C[U_Q^+]\cong \A/\mathcal{I}$, where $\mathcal{I}$ is the ideal generated by slice maps 
$ (\omega\otimes \iota)\left((T\otimes 1)U^{\otimes k}-U^{\otimes l}(T \otimes 1)\right)$, for $T\in \CC_Q(k,l)$ and $\omega\in B(H^{\otimes k},H^{\otimes l})^*$. The claim will then follow from 
Theorem~\ref{tmain}.
By definition, $\C[U_Q^+]$ can be written as $\A/\mathcal{L}$ where $\mathcal{L}$ is the ideal generated by the relations
$$UF^{t}U^{t}(F^{-1})^{t}=1,\ \ F^tU^{t}(F^{-1})^{t}U=1,\ \ U_{12}=0,\ \ U_{21}=0.   $$
We already know that the ideal $\I$ contains slices of the first two relations, since we showed in Lemma~\ref{lorth} that they correspond to the slice maps $$(\omega_{1,ij}\otimes \iota)
\left((R\otimes \iota)-U^{\otimes 2}(R\otimes \iota) \right)_{i,j},\ \ \ \ (\omega_{ij,1}\otimes \iota)\left((R^{*}\otimes \iota)U^{\otimes 2}-(R^{*}\otimes \iota) \right)_{i,j}.$$
The other two relations correspond to $(\omega_{ij}\otimes \iota)\left( (p\otimes \iota)U-U(p\otimes \iota)\right)_{i,j}$. Hence $\mathcal{L}\subseteq \I$.
The opposite inclusion follows analogously to the second part of the proof of Lemma~\ref{lorth}.

\end{example}

\bigskip

\end{document}